\documentclass[10pt]{amsart}
\usepackage{amsmath,amssymb,enumerate}

\newtheorem{thm}{Theorem}[section]
 \numberwithin{equation}{section} 
 \numberwithin{figure}{section} 
 \theoremstyle{plain}
 \theoremstyle{plain}    
 \theoremstyle{plain}    
 \newtheorem{prop}[thm]{Proposition} 
 \newtheorem{defi}[thm]{Definition}
 \theoremstyle{plain}    
 \newtheorem{lem}[thm]{Lemma} 
 \theoremstyle{remark}
 
 \theoremstyle{definition}

\newtheorem{mthm}{Theorem}

\newcommand{\N}{\mathbb{N}}
\newcommand{\R}{\mathbb{R}}
\newcommand{\C}{\mathbb{C}}
\newcommand{\D}{\mathbb{D}}

\newcommand{\Cc}{\mathcal{C}}
\newcommand{\f}{\varphi}
\newcommand{\p}{\psi}
\newcommand{\vep}{\varepsilon}
\newcommand{\Ec}{\mathcal{E}}
\newcommand{\EcX}{\mathcal{E}(X,\omega)}
\newcommand{\Fc}{\mathcal{F}}
\newcommand{\Uc}{\mathcal{U}}
\newcommand{\Hc}{\mathcal{H}}
\newcommand{\Sc}{\mathcal{S}}
\newcommand{\ind}{{\bf 1}}

\newcommand{\setdef}{\ \big\vert \ }
\newcommand{\Capa}{{\rm Cap}}
\newcommand{\Capo}{{\rm Cap}_{\omega}}
\newcommand{\CapBT}{{\rm Cap}_{\rm BT}}

\newcommand{\Capis}{{\rm Cap}_{\psi}}

\newcommand{\MA}{\mathrm{MA}\,}
\newcommand{\vol}{{\rm Vol}}

\newcommand{\tr}{{\rm tr}}
\newcommand{\psh}{{\rm PSH}}

\begin{document}
\title[MA equations on quasi-projective varieties]{Complex Monge-Amp\`ere equations on quasi-projective varieties}

\date{\today \\
 The authors are partially supported by the french ANR project MACK} 

\author{E. Di Nezza} 

\address{Institut Math\'ematiques de Toulouse \\
Universit{\'e} Paul Sabatier\\ 31062 Toulouse\\ France}

\email{eleonora.dinezza@math.univ-toulouse.fr}

\author{Chinh H. Lu}
\address{Chalmers University of Technology \\ Mathematical Sciences\\
412 96 Gothenburg\\ Sweden}

\email{chinh@chalmers.se}

\begin{abstract}
We introduce generalized Monge-Amp\`ere capacities and use these to study complex Monge-Amp\`ere equations 
whose right-hand side is smooth outside a divisor. 
We prove, in many cases, that there exists a unique normalized solution  which is smooth outside the divisor.
\end{abstract}

\maketitle

\section{Introduction}

Let $(X,\omega)$ be a compact K\"{a}hler manifold of complex 
dimension $n$ and let $D$ be a divisor on $X$. Let $f$ be a 
non-negative 
function  such that $\int_X f\omega^n=\int_X \omega^n$.
Consider the following  complex Monge-Amp\`ere equation 
\begin{equation}\label{eq: intro 1}
 (\omega+dd^c \varphi)^n=f\omega^n. 
\end{equation}

When $f$ is smooth and positive on $X$, it follows from the seminal work of Yau \cite{Y78}
that (\ref{eq: intro 1}) admits a unique normalized smooth solution $\f$  such that 
$\omega+dd^c \f$ is a K\"ahler form. Recall that this result solves in particular the Calabi conjecture
and allows to construct Ricci flat metrics on $X$ whenever $c_1(X)=0$.

It is very natural to look for a similar result when $f$ is merely smooth and positive on the complement
of $D$, e.g. when studying Calabi's conjecture on quasi-projective manifolds (see e.g. \cite{TY87,TY90,TY91} and \cite{H12}) for recent developments). The study of conical K\"ahler-Einstein metrics 
(K\"ahler-Einstein metrics in  the complement of a divisor 
with  a precise  behavior near $D$) has played a major role in the resolution of the Yau-Tian-Donaldson
conjecture for Fano manifolds (see \cite{Don12},\cite{DS12},\cite{CDS1,CDS2,CDS3},\cite{T12}).

However no systematic study of the regularity of solutions to such complex Monge-Amp\`ere equations 
has ever been done, this is the main goal of this article.
It follows from \cite{GZ07} that (\ref{eq: intro 1}) has 
 a unique (up to an additive constant) solution in  the finite energy class $\Ec(X,\omega)$. 
We say that the solution is normalized if  $\sup_X \f=0$. The problem thus boils down to showing that such a normalized
solution is smooth in $X \setminus D$ and understanding its asymptotic behavior along $D$.

As in the classical case of Yau \cite{Y78} the main difficulty is in establishing a priori 
$\Cc^0$ bounds. Since, in general the solution $\varphi$ is unbounded,
the idea is to bound $\varphi$ from below by some (singular) 
$\omega$-psh function. 

Our first main result shows that the solution $\f$ is smooth in $X\setminus D$ when $f$ satisfies the mild condition $\Hc_f$:
 \begin{eqnarray*}\label{eq: smooth condition}
\ \ \quad f= e^{\psi^+-\psi^-}, \ \psi^{\pm}\ {\rm are\ quasi \ plurisubharmonic\ on }\ X, \ 
 \ \psi^-\in L^{\infty}_{\rm loc}(X\setminus D).
\end{eqnarray*}

Let us stress that $D$ is here an arbitrary divisor.
\begin{mthm}\label{thm: main general}
{\it Assume that $0<f\in \Cc^{\infty}(X\setminus D)$ satisfies Condition $\Hc_f$. Then the solution $\varphi$ is also smooth on $X\setminus D$.}
\end{mthm}

In Theorem \ref{thm: main general}, the density $f$ is only in 
$L^1(X)$ and there is no regularity assumption on $D$. Hence we do not have any information about the behavior 
of  $\f$ near $D$. If we assume more regularity on $f$ and $D$, we will
get more precise $\Cc^0$-bounds. 

Assume that $D= \sum_{j=1}^N D_j$
is a simple normal crossing divisor (snc for short). For each $j=1,...,N$, let  $L_j$ be the holomorphic line bundle defined by 
$D_j$. Let $s_j$ be  a holomorphic section of $L_j$ such that 
$D_j=\{s_j= 0\}$. 
Fix a hermitian metric $h_j$ on $L_j$ such  that $|s_j|:=\vert s_j\vert_{h_j} \leq 1/e$. 

We say that $f$ satisfies Condition $\Sc(B,\alpha)$ for some constants $B>0,\alpha>0$ if 
\begin{equation*}\label{eq: cond f log}
f\leq \frac{B}{\prod_{j=1}^N|s_j|^{2}
(-\log \vert s_j\vert)^{1+\alpha}}\ .
\end{equation*}

\begin{mthm}\label{thm: main log}
{\it Assume that $f\leq e^{-\phi}$ for some quasi-plurisubharmonic function $\phi$. Then for each $a>0$ there exists $A>0$ depending 
on $\int_X e^{-2\f/a}\omega^n$ such that
$$
\f\geq a \phi -A.
$$
More precisely, if $f$ satisfies Condition $\Sc(B,\alpha)$ for some  
$\alpha>0$, $B>0$, then the following holds:
\begin{itemize}
\item[(a)] if $\alpha>1$ then $\varphi$ is continuous on $X$, 
$\varphi\geq -C$, with $C=C(B,\alpha)$.
\item[(b)] if $\alpha=1$ then there exists $A_1, A_2 >0$ depending on $B$ such that 
$$
\varphi \geq \sum_{j=1}^N -A_1\left[\log (-\log \vert s_j\vert + A_2)\right],
$$
\item[(c)] if $\alpha\in (0,1)$ then for each 
$\beta\in (1-\alpha,1)$ and $a>0$ there exists $A>0$ depending on 
$a,\alpha,\beta,B$ such that 
$$
\varphi \geq \sum_{j=1}^N -a(-\log \vert s_j\vert)^{\beta} -A. 
$$ 
\end{itemize}
}
\end{mthm}
\medskip

\noindent{\bf Remark.}{\it\  It follows from Skoda's theorem \cite{Sk72} that $\int_X e^{-2\f/a}\omega^n$ is finite for all $a>0$, since $\f\in \Ec(X,\omega)$ has zero Lelong number at all points \cite{GZ07}.}

\medskip

When the behavior of $f$ near the divisor $D$ looks exactly like
$$
\frac{1}{\prod_{j=1}^N|s_j|^2 |\log |s_j||^{1+\alpha}} , \ \alpha \in (0,1]
$$
we show in Proposition \ref{prop: asymptotic 1} and Proposition \ref{prop: asymptotic 2} that $\f(x)$ converges to $-\infty$ as $x$ approaches $D$ with precise rates. In particular there is no bounded 
solution to (\ref{eq: intro 1}).

When $f \in L^p(\omega^n)$ for some $p>1$, it follows from the work of Ko{\l}odziej \cite{Kol98}
that the solution of (\ref{eq: intro 1}) is actually uniformly bounded (and even H\"older continuous) on the whole of $X$. 

In our result, the density $f$ is merely in $L^1$. 
The first part of Theorem \ref{thm: main log} says that when 
$\alpha>1$ the solution is continuous on $X$. Ko{\l}odziej's result 
\cite[Theorem 2.5.2]{Kol98} also applies when $f$ satisfies 
$\Sc(B,\alpha)$ for $\alpha>n$ but can not be applied to a density $f$ as above if $\alpha\leq n$.

Observe furthermore that $\alpha=1$ is a critical exponent as is 
easily seen when $n=1$. In any dimension, when $f$ has singularities of Poincar\'e type,
$$
\frac{1/C}{\prod_{j=1}^N|s_j|^2 |\log |s_j||^2} \leq f(z) \leq  \frac{C}{\prod_{j=1}^N|s_j|^2 |\log |s_j||^2}
$$
along $D$ we show in Section \ref{sect: asymp} that the solution is locally uniformly bounded on compact subsets of $X\setminus D$ and goes to $-\infty$ along $D$ with a certain rate.  If moreover $f$ has a "very precise" behavior near $D$  it follows from the recent work of Auvray (see \cite{Auv}) that $\f$ goes to $-\infty$ along $D$ like $\sum_{j=1}^N-\log(-\log|s_j|)$. We stress that this condition is very restrictive while in our result we only need a very weak condition on the density. Recall also that in \cite{TY87} the authors constructed "almost complete" K\"ahler Einstein metrics of negative Ricci curvature on $X\setminus D$. In this case the $\Cc^0$ estimate follows easily from the maximum principle.   
\medskip

In order to prove the $\Cc^0$-estimate we follow
and generalize Ko{\l}odziej's approach. 
We introduce and study the $\psi$-Capacity of a Borel subset $E\subset X$,
\begin{equation*}\label{eq: capa two bound def}
\Capa_{\psi} (E):=\sup\left\{\int_E (\omega+dd^c u)^n
\  \big \vert\  u\in \psh(X,\omega),\ \psi-1\leq u\leq \psi \right \}
\end{equation*}
where $\psi\in \psh(X,\omega)$ and here $(\omega+dd^c u)^n$ is the
nonpluripolar Monge-Amp\`ere measure of $u$ (see Section \ref{sect: basic properties} for the definition).
When $\psi$ is constant, $\psi\equiv C$, we recover the 
Monge-Amp\`ere capacity,
$$
\Capo=\Capa_{C}.
$$

A similar notion has been studied in \cite{CKZ05} 
in a local context. These generalized  capacities are interesting 
for themselves. In this paper we only need some of their properties and refer the reader to \cite{DL1} for a more systematic study.

\medskip

One of the advantages of the Ko{\l}odziej's approach for the $\Cc^0$ estimates is that it also works in the case of  
semipositive and big classes as shown in \cite{BGZ08}, \cite{EGZ09}
and \cite{BEGZ10}. 
Thus it is not surprising that our method is still valid in this situation.
  
Let $\theta$ be a smooth  semipositive form on $X$ such that $\int_X \theta^n>0$. Let $f$ be a non-negative function  such that  $\int_X f\omega^n = \int_X \theta^n$. Consider the following degenerate complex Monge-Amp\`ere equation
\begin{equation}\label{eq: MAeq semipositive}
(\theta + dd^c \varphi)^n = f\omega^n .
\end{equation}

It follows from \cite{BBGZ13} that (\ref{eq: MAeq semipositive}) admits a unique normalized solution $\f\in \Ec(X,\theta)$. As in 
the K\"ahler case, it is 
interesting to investigate the regularity properties of 
$\f$ if we know that the density $f$ is smooth, strictly positive outside a divisor $D$ and verifies Condition $\Hc_f$. We can not expect $\f$ to be smooth on 
$X\setminus D$ since $\theta$ may be zero somewhere there. Our result below shows that the solution is smooth on $X\setminus (D\cup E)$, where $E$ is an effective simple normal crossing  divisor on $X$ such that $\{\theta\}-c_1(E)$ is ample.

\begin{mthm}\label{thm: main semi}
{\it Let $(X,\omega)$ be a compact K\"{a}hler manifold of complex 
dimension $n$ and $D$ be an arbitrary divisor on $X$. 
Let $E$ be an effective snc divisor on $X$, and $\theta$ be a smooth  semipositive form on $X$ such that $\int_X \theta^n>0$ and $\{\theta\}-c_1(E)$ is ample.
Assume that $0< f\in \Cc^{\infty}(X\setminus D)$ satisfies Condition $\Hc_f$. Let $\varphi$ be  the unique normalized solution to equation (\ref{eq: MAeq semipositive}). Then $\varphi$ is smooth on $X\setminus (D\cup E)$.} 
\end{mthm}

\medskip

\noindent{\bf Remark.} 
{\it The condition we impose on $\{\theta\}$ is natural in studying K\" ahler Einstein metrics on singular varieties (see \cite{BG13}). }

\medskip

Let us say some words about the  organization of the paper. 
In Section \ref{sect: basic properties}, we introduce the generalized $\psi$-Capacity,
and establish their basic properties. 
The proof of Theorem \ref{thm: main general} will be given in Section \ref{sect: smooth  general}. We provide some volume-capacity estimates in Section \ref{sect: vol cap}. We then use these  to  prove Theorem  \ref{thm: main log} and discuss about the asymptotic behavior of solutions near the divisor in Section \ref{sect: proof main log}. Finally we consider the case of semipositive and big classes in Section \ref{sect: semi}.

\bigskip

\noindent {\bf Acknowledgments.} It is our pleasure to thank our advisors
Vincent Guedj and Ahmed Zeriahi for providing constant help, many suggestions and encouragements. We also thank  Robert Berman and Bo Berndtsson for very useful comments. We are indebted to S\' ebastien Boucksom and Henri Guenancia for a very careful reading of a preliminary version of this paper, for their suggestions which improve the presentation of the paper.

\section{Preliminaries}\label{sect: basic properties} 
Let $(X,\omega)$ be a compact K\"ahler manifold.  
We first recall basic facts about finite energy classes of 
$\omega$-psh functions on $X$. The reader can find  more details
about these  in \cite{GZ07}.

\subsection{Finite energy classes}
\begin{defi}
We let $\psh(X,\omega)$ denote the class of $\omega$-plurisubharmonic functions ($\omega$-psh for short) on $X$, i.e. the class of functions 
$\f$ such that locally  $\f= \rho+ u$, where $\rho$ is a local potential of $\omega$ and $u$ is a plurisubharmonic function.
\end{defi}

Let $\f$ be some (unbounded) $\omega$-psh function on $X$ 
and consider $\f_j:=\max(\f, -j)$ the canonical approximation by 
bounded $\omega$-psh functions. 
It follows from \cite{GZ07} that 
$$
{\ind}_{\{\f_j >-j \}} (\omega+dd^c \f_j)^n
$$ 
is a non-decreasing sequence of Borel measures. 
We denote by $(\omega+dd^c\varphi)^n$ (or $\MA(\varphi)$ for short if 
$\omega$ is fixed and no confusion can occur) this limit:
$$
\MA(\varphi)=(\omega+dd^c \varphi)^n = \lim_{j\to+\infty}
{\ind}_{\{\f_j >-j \}} (\omega+dd^c \f_j)^n.
$$
It was shown in \cite{GZ07} that the Monge-Amp\`ere measure 
$\MA(\varphi)$ puts no mass on pluripolar sets. This is the non-pluripolar part of the Monge-Amp\`ere of $\f$. Note that its total mass $ \MA(\f)(X)$ can take value in $\left[0, \int_X\omega^n \right]$.
\begin{defi}
We let $\Ec(X,\omega)$ denote the class of $\omega$-psh function having full Monge-Amp\`ere mass:
$$
\EcX:= \left\{\f\in \psh(X,\omega) \setdef \int_X \MA(\varphi) =\int_X \omega^n\right\}.
$$
\end{defi}
Let us stress that $\omega$-psh functions with full 
Monge-Amp{\`e}re mass have mild singularities. Indeed, it was shown
in \cite[Corollary 1.8]{GZ07} that 
$$
\nu(\f,x)=0, \forall \f\in \EcX,\  x\in X.
$$
 
We also recall that, for every 
$\f\in\Ec(X,\omega)$ and $\psi\in \psh(X,\omega)$, the 
\emph{generalized comparison principle} holds (see \cite[Corollary 2.3]{BEGZ10}), namely 
$$
\int_{\{\f<\psi\}} (\omega+dd^c  \psi)^n \leq \int_{\{\f<\psi\}} (\omega+dd^c \f)^n.
$$
Let $\chi:\R^-\to\R^-$ be an increasing function such 
that $\chi(0)=0$ and $\chi(-\infty)=-\infty$. 
\begin{defi} 
{\it Let $\Ec_\chi(X,\omega)$ denote the set of $\omega$-psh functions with finite $\chi$-energy,
$$\Ec_{\chi}(X,\omega):= \{\f\in \Ec(X,\omega) \;\,  \setdef \;\, \chi(-|\f|)\in L^1(\MA(\f)) \}.$$ }
\end{defi} 

For $p>0$, we  use the notation
$$
{\Ec}^p(X,\omega):=\Ec_{\chi}(X,\omega),
\text{ when } \chi(t)=-(-t)^p.
$$

\subsection{The $\p$-Capacity}
\begin{defi}
Let $\psi\in \psh(X,\omega)$. We define the $\psi$-Capacity of a Borel subset $E\subset X$ by
$$
\Capis(E):= \sup\left\{ \int_E \MA (u) \ \vert \ u\in \psh(X, \omega) , \ \psi-1\leq u\leq \psi\right\}.
$$
\end{defi}

Then the  Monge-Amp\`ere capacity corresponds to $\psi \equiv {\rm constant}$ (see \cite{BT82}, \cite{Kol03}, \cite{GZ05}). We list below some basic properties of the $\p$-Capacity.
\begin{prop}\label{prop: basic prop of psi capacity}
\begin{itemize}
\item[(i)] If $E_1\subset E_2\subset X$ then $\Capis(E_1)\leq \Capis(E_2)$ .

\item[(ii)]  If $E_1, E_2,...$ are Borel subsets of $X$ then 
$$
\Capis\left(\bigcup_{j=1}^{\infty} E_j\right)\leq \sum_{j=1}^{+\infty} \Capis (E_j) .
$$

\item[(iii)] If $E_1\subset E_2\subset ...$ are Borel subsets of $X$ then
$$
\Capis \left(\bigcup_{j=1}^{\infty} E_j\right) = \lim_{j\to+\infty} \Capis (E_j) .
$$
\end{itemize}
\end{prop}
The following results are elementary and important for the sequel. We stress that these results still hold in the case when $\omega$ is merely semipositive and big rather than K\" ahler.

\begin{lem}\label{lem: right continuous}
Let $\psi\in \psh(X,\omega)$ and $\f\in \EcX$. Then the function
$$
H(t):= \Capis(\{\f<\psi-t\}) , \ t\in \R ,
$$
is right-continuous and $H(t)\to 0$ as $t\to+\infty$.
\end{lem}
\begin{proof}
The right-continuity of $H$ follows from (iii) of Proposition \ref{prop: basic prop of psi capacity}. Let us prove the second 
statement. We can assume that $\psi\leq 0$ on $X$. Fix $v\in \psh(X,\omega)$
such that $\p-1\leq v \leq \p$.  
We apply the comparison principle to obtain
\begin{eqnarray*}
\int_{\{\f<\psi-t\}} \MA(v) \leq  \int_{\{\f<v-t+1\}}\MA(v)  
\leq  \int_{\{\f<-t+1\}} \MA(\f).
\end{eqnarray*}
The last term goes to zero as $t$ goes to 
$+\infty$ since $\f\in \EcX$.
\end{proof}

\begin{lem}\label{lem: classical cap and cap psi}
Let $(X,\omega)$ be a compact K\" ahler manifold and $\psi\in \psh(X,\omega/2)$. Then we have 
$$
\Capa_{\omega/2}(E) \leq \Capis(E),
$$
where $\Capa_{\omega/2}$ is the Monge-Amp\`ere Capacity with respect to the K\" ahler metric $\omega/2$ introduced in \cite{Kol03} and studied in \cite{GZ05}, and 
$\Capis$ is the generalized $\psi$-Capacity with respect to the 
K\"ahler metric $\omega$. 
\end{lem}
We stress that the above result insures $\Capis(E)>0$ for any Borel subset $E$ which is not pluripolar.
\begin{proof}
Let $u\in \psh(X,\omega/2)$ be such that $-1\leq u\leq 0$. Then $\varphi:=\psi+u$ is a candidate defining $\Capis$. Using the definition of the Monge-Amp\`ere meausure it is not difficult to see that
$$
\int_E (\omega/2 + dd^c u)^n \leq \int_E (\omega+dd^c \varphi)^n\leq \Capis(E),
$$
and taking the supremum over all $u$ we get the result.
\end{proof}
The following result generalizes Lemma 2.3 in \cite{EGZ09}.
\begin{prop}\label{prop: EGZ09 SKE generalized}
Let $\varphi\in \Ec(X,\omega)$, $\psi\in \psh(X,\omega)$. Then  for all $t>0$ and $0\leq s\leq 1$ we have 
$$
s^n \Capis(\{\varphi<\psi-t-s\})\leq \int_{\{\varphi<\p-t\}}\MA(\varphi).
$$
\end{prop}
                                                                                                                                                                                                         \begin{proof}                                                                                                                                                                                       Let $u\in \psh(X,\omega)$ such that $\psi-1\leq u\leq \psi$. 
Observe the following trivial inclusion                                                                                                                                                                                                $$
\{\varphi<\psi-t-s\}\subset \{\varphi<s u +(1-s)\psi-t\} 
\subset \left\{\varphi<\psi-t \right\} .                                                                                                                                                                                                     
$$
It thus follows from the {\it generalized comparison principle} 
(see \cite[Corollary 2.3]{BEGZ10}) that
                                                                                                                                                                                                                                                                                                                                    \begin{eqnarray*}
s^n \int_{\{\varphi<\psi-t-s\}}\MA(u) 
& \leq & \int_{\{\varphi<\p-t-s\}}\MA(s u+ (1-s)\psi) \\
&\leq & \int_{\{\varphi<s u +(1-s)\p-t\}}\MA(s u+ (1-s)\psi)\\
&\leq & \int_{\{\varphi<\p-t\}}\MA(\varphi).
\end{eqnarray*}                                                                                                                                                                                        
By taking the supremum over all candidates $u$ we get the result.
\end{proof}

\section{Smooth solution in a general case}\label{sect: smooth general}

In this section we prove Theorem \ref{thm: main general}.  The most difficult part is the $\Cc^0$ estimate which is followed by a simple observation: if $\varphi\in \EcX$, $\sup_X \f=0$ is such that $\MA(\f) \leq e^{-\phi}\omega^n$, for some quasi-psh function $\phi$, then $\varphi$ is bounded from below by $a\phi-A$, for some positive constants $a,A$.

\subsection{Uniform estimate} In this subsection we assume that $0\leq f\in L^1(X)$ is such that $\int_X f\omega^n= \int_X \omega^n$.
Let $\f\in \Ec(X,\omega)$ be the unique normalized solution to 
\begin{equation}\label{eq: MA kah uniform}
(\omega+dd^c \f)^n =f\omega^n.
\end{equation}
Here we normalize $\f$ such that $\sup_X \f=0$. We prove the following $\Cc^0$ estimate:

\begin{thm}\label{thm: general uniform}
Assume that $f\leq e^{-\phi}$ for some quasi-plurisubharmonic function
$\phi$. Let $\f\in \EcX$ be the unique normalized solution to (\ref{eq: MA kah uniform}).
Then for any  $a>0$, there exists $A>0$ depending on 
$\int_X e^{-2\f/a } \omega^n$ such that
$$
\varphi \geq a \phi -A.
$$
Moreover, if $\phi$ is  bounded in a compact subset $K\subset X$ then $\f$ is continuous on $K$.
\end{thm}

\begin{proof}
We can assume that $\phi\leq 0$. Observe that it is enough to prove Theorem \ref{thm: general uniform} for $a>0$ small enough. Fix $a>0$
such that $\p:=a\phi$ belongs to $\psh(X,\omega/2)$.
It follows from Lemma \ref{lem: classical cap and cap psi} that 
$\Capo \leq 2^n\Capa_{\omega/2}\leq 2^n \Capis$. 
Fix $s\in [0,1]$, $t>0$ and apply Proposition \ref{prop: EGZ09 SKE generalized} to get
\begin{equation}\label{eq: general uniform 1}
s^n\Capis(\f<\p-t-s) \leq \int_{\{\f<\p-t\}} \MA(\f). 
\end{equation}
By assumption on $f$ we have
$$
\int_{\{\f<\p-t\}} \MA(\f) \leq \int_{\{\f<\p-t\}} e^{-\f/a}
e^{\p/a}\MA(\f)\leq \int_{\{\f<\p-t\}} e^{-\f/a}\omega^n.
$$
It follows from \cite{GZ05} that 
$$
\vol_{\omega} \leq \exp{\left(\frac{-C_1}{\Capo^{1/n}}\right)}.
$$
Thus using H\"older inequality we get from (\ref{eq: general uniform 1}) that
$$
s^n\Capis(\f<\p-t-s) \leq C_2 \left(\Capo(\f<\p-t)\right)^2 \leq C_3
\left(\Capis(\f<\p-t)\right)^2,
$$
where $C_3$ depends only on $\int_X e^{-2\f/a}\omega^n$. Now, consider the following function
$$
H(t)=\left[\Capa_{\psi}(\{\varphi<\psi-t\})\right]^{1/n},\ t>0.
$$
By the arguments above we get 
$$
s H(t+s)\leq C_4 H(t)^2,  \ \forall t>0, \forall s\in [0,1],
$$
where $C_4>0$ depends only on $\int_X e^{-2\f/a}\omega^n$. It follows 
from Lemma \ref{lem: right continuous} that $H$ is right-continuous 
and $H(+\infty)=0$. Thus  by \cite[Lemma 2.4]{EGZ09} we get $\varphi \geq \psi-C_5$, where $C_5$ only depends on 
$\int_X e^{-2\f/a}\omega^n$.

Now, assume that $\phi$ is bounded on a compact subset $K\subset X$. 
Set $\p:=a\phi$ as above. Let us prove that $\f$ is continuous on $K$. For convenience, we normalize $\varphi$ so that $\sup_X \varphi=-1$. Let $0\geq \varphi_j$ be a sequence of continuous 
$\omega$-psh functions on $X$ decreasing to $\varphi$.
Fix $\lambda\in (0,1)$. For each $j\in \N$ set   
$$
\psi_j:=\lambda \varphi_j + (1-\lambda)\psi - (1-\lambda) A - 2.
$$ 
Then $\p_j$ belongs to $\psh(X,\frac{1+\lambda}{2} \omega)$ and 
 $\psi_j \leq \varphi_j - 2$. Set 
$$
H_j(t):= \left[\Capa_{\psi_j}\left(\{\varphi<\psi_j-t\}\right)\right]^{1/n}, t>0.
$$
We can argue as above and use Proposition \ref{prop: EGZ09 SKE generalized} to get 
$$
s H_j(t+s)\leq C_1H_j(t)^2, \ \forall t>0, \ \forall s\in [0,1],
$$
where $C_1>0$ depends on $\int_X e^{-2\f/(1-\lambda)a }$. Let
 $\chi: \R^-\rightarrow \R^-$ be an increasing convex weight such that
 $\chi(0)=0$, $\chi(-\infty)=-\infty$ and $\varphi\in \Ec_{\chi}(X,\omega)$. 
By the comparison principle we also get
\begin{eqnarray*}
\Capa_{\psi_j}(\varphi<\psi_j) &\leq & \int_{\{\varphi<\psi_j+1\}} \MA
(\varphi)\leq \int_{\{\varphi<\varphi_j-1\}} \MA(\varphi)\\
&\leq &\frac{1}{-\chi(-1)}\int_X (-\chi\circ(\varphi-\varphi_j)) f\omega^n.
\end{eqnarray*}

The latter converges to $0$ as $j\to +\infty$, since $\varphi_j$ decreases to $\varphi$. Thus for $j$  big enough we have 
$H_j(0)\leq 1/(2C_1)$. 
It then follows from \cite[Remark 2.5]{EGZ09} that 
 $H_j(t)=0$ if $t\geq t_{\infty}$ where $t_{\infty}\leq C_2H_j(0)$ 
 and $C_2$ depends on $C_1$.  We then get
 $$
 \varphi\geq \lambda \varphi_j+(1-\lambda) 
 \psi -C_2 H_j(0).
 $$
 Now, letting $j\to+\infty$, we get
 $$
 \lim_{j\to+\infty}\inf_K (\varphi-\varphi_j) \geq  (\lambda-1)\sup_K \vert\psi\vert .
 $$
Finally, letting $\lambda\to 1$ we get the continuity of $\varphi$ on $K$. 
\end{proof}

\subsection{Laplacian estimate}
The following a priori estimate  generalizes \cite{Paun08}.

\begin{thm}\label{thm: C2 estimate}
Let $\mu$ be a positive measure on $X$ of the form $\mu=e^{\psi^+-\psi^-}\omega^n$ where $\psi^+,\psi^-$ are smooth on $X$. Let $\varphi\in \Cc^{\infty}(X)$  be such that $\sup_X\varphi = 0$ and
$$
(\omega+dd^c \varphi)^n=e^{\psi^+-\psi^-}\omega^n .
$$
Assume given a constant  $C>0$ such that
$$
dd^c \psi^{\pm}\geq -C\omega , \ \  \sup_X \psi^+\leq C .
$$
Assume also that the holomorphic bisectional curvature of $\omega$ is bounded from below by  $-C$. 
Then there exists $A>0$ depending on 
$C$ and  $\int_X e^{-2(4C+1)\f}\omega^n$ such that 
$$
0\leq n+\Delta_{\omega} \varphi \leq A e^{-2\psi^-}.
$$
\end{thm}
We follow the lines in Appendix B of \cite{BBEGZ12}. We recall the following result:
\begin{lem}\label{lemm1}
Let $\alpha, \beta$ be positive $(1,1)$-forms. Then
$$
n \left(\frac{\alpha^n}{\beta^n}\right)^{\frac{1}{n}}\leq \tr_{\beta}(\alpha)\leq n \left(\frac{\alpha^n}{\beta^n}\right)\cdot \left( \tr_{\alpha}(\beta)\right)^{n-1}. 
$$
\end{lem}
\begin{proof}[Proof of Theorem \ref{thm: C2 estimate}]
Set $\omega_{\varphi}:=\omega+dd^c\varphi$. Since the holomorphic bisectional curvature of $\omega$ is bounded from below by $-C$, it 
follows from  Lemma 2.2 in \cite{CGP13} that 
\begin{equation}\label{eq: C2 estimate 1}
\Delta_{\omega_\varphi}\log \tr_{\omega}(\omega_\varphi)\geq \frac{\tr_\omega (dd^c \p^+ -dd^c \p^-)}{\tr_\omega(\omega_\varphi)}-
C\tr_{\omega_\varphi}(\omega).
\end{equation}
Since $dd^c \p^+\geq -C\omega$, using the trivial inequality 
$n\leq \tr_{\omega}(\omega_\f)\tr_{\omega_\f}(\omega)$ 
 we thus get from (\ref{eq: C2 estimate 1}) that
\begin{eqnarray}\label{eq: C2 estimate 2}
\Delta_{\omega_\varphi}\log \tr_{\omega}(\omega_\varphi)&\geq & -\frac{\tr_\omega (C\omega+dd^c\psi^-)}{\tr_\omega(\omega_\varphi)}-C \tr_{\omega_\varphi}(\omega) \nonumber \\
&\geq & -2C\tr_{\omega_\varphi}(\omega)-\frac{\Delta\psi^-}{\tr_\omega(\omega_\varphi)}\, .
\end{eqnarray}
By assumption we have $0\leq C\omega +dd^c\psi^-\leq \tr_{\omega_\f}
(C\omega+dd^c \psi^-)\omega_\f$. Applying $\tr_{\omega}$ to the 
previous inequality  yields 
$$
C n+\Delta\psi^-\leq (C \tr_{\omega_{\f}}(\omega)+\Delta_{\omega_\f}\psi^-)\tr_\omega (\omega_\f),
$$ 
and hence
$$
-\Delta \psi^- \geq -(C \tr_{\omega_{\f}}(\omega)+\Delta_{\omega_\f}\psi^-)\tr_\omega (\omega_\f).
$$
Thus, plugging this into (\ref{eq: C2 estimate 2}) we obtain 
\begin{eqnarray}\label{eq: C2 estimate 3}
\Delta_{\omega_\varphi}\log \tr_{\omega}(\omega_\varphi)\geq -3C \tr_{\omega_\varphi}(\omega)-\Delta_{\omega_\f}\psi^-.
\end{eqnarray}
We want now to apply the maximum principle to the function 
$$
H:= \log \tr_{\omega}(\omega_\f)+ 2\psi^- -(1+4C)\f,
$$
Let $x_0\in X$ be such that
$H$ achieves its maximum on $X$ at  $x_0$. Then at $x_0$ we get
\begin{eqnarray*}
0\geq \Delta_{\omega_\varphi}H &\geq & \tr_{\omega_\varphi}(\omega)-n(1+4C) .
\end{eqnarray*}
Furthermore, by Lemma \ref{lemm1} we get
$$
\tr_{\omega}(\omega_\f)(x_0)\leq n e^{\psi^+-\psi^-}(x_0) \left(\tr_{\omega_\varphi}(\omega)\right)^{n-1}(x_0)\leq A_1 e^{\psi^+-\psi^-}(x_0) ,
$$ 
and hence, since $\sup_X \psi^+\leq C$,
$$
\log \tr_{\omega}(\omega_\f)(x_0)\leq \log A_1 +\psi^+(x_0)-\psi^-(x_0)\leq A_2-\psi^-(x_0)\, .
$$
It follows that
$$
H(x)\leq H(x_0)\leq A_3 +\p^-(x_0)- (1+4C) \f(x_0).
$$ 
By assumption and the $\Cc^0$ estimate in Theorem \ref{thm: general uniform} we have $\f\geq a \psi^--A_4$, where $a=1/(4C+1)$ and $A_4$ depends on $C$ and $\int_X e^{-2\f/a}\omega^n$.  Thus 
$$
\log \tr_{\omega}(\omega_\f)\leq A_5-2\psi^-.
$$
We finally infer as desired 
$$
\tr_{\omega}(\omega_{\f})\leq A_6 e^{-2\psi^- }.
$$

\end{proof}
We are now ready to prove Theorem \ref{thm: main general}.

\subsection{Proof of Theorem \ref{thm: main general}}\label{sect: proof main general}
Let $\f\in \EcX$ be the unique normalized solution to 
$$
(\omega+dd^c \f)^n =f\omega^n.
$$
By assumption we can write $\log f= \psi^+-\psi^-$, where 
$\psi^{\pm}$ are quasi psh functions on $X$, $\psi^-$ is 
locally bounded on $X\setminus D$, and there is 
a uniform constant $C>0$ such that
$$
dd^c \psi^{\pm}\geq -C\omega, \ \sup_X \psi^+\leq C.
$$ 
We now approximate $\psi^{\pm}$ by using Demailly's regularization operator $\rho_{\varepsilon}$. We recall the construction:
if $u$ is a quasi-psh function on $X$ and $\varepsilon>0$ we set
$$
\rho_{\varepsilon}(u)(z):= \frac{1}{\varepsilon^{2n}}
\int_{\zeta\in T_{X,z}} u({\rm exph}_z(\zeta)) \chi\left(\vert\zeta
\vert^2/\varepsilon^2\right)d\lambda(\zeta).
$$
Here $\chi\in \Cc^{\infty}(\R)$ is a cut-off function supported in $[-1,1]$, $\int_{\R} \chi(t) dt =1$,  and 
$$
{\rm exph}: TX \rightarrow X ,\ \  \zeta \mapsto {\rm exph}_z(\zeta)
$$
is the formal holomorphic part of the Taylor expansion of the exponential map defined by  the metric $\omega$. For more details, see \cite{Dem92}. Observe that by Jensen's inequality,
$\rho_{\vep}(e^u)\geq e^{\rho_{\varepsilon}(u)}$.
Applying this smoothing regularization 
to $\psi^{\pm}$ we get, for $\varepsilon>0$ small enough,
$$
dd^c \rho_{\varepsilon}(\psi^{\pm})\geq -C_1\omega, \ \ e^{\rho_{\vep}(\p^+-\p^-)} \leq e^{-\rho_{\vep}(\p^-)+C_1},
$$  
where $C_1$ depends on $C$ and the Lelong numbers of the currents
$C\omega+dd^c \psi^{\pm}$.  
Now, for each $\varepsilon>0$, let $\f_{\vep}\in
\Cc^{\infty}(X)$ be the unique normalized solution to
$$
(\omega+dd^c \f_{\vep})^n= c_{\vep} 
e^{\rho_{\vep}(\psi^+)-\rho_{\vep}(\psi^-)}\omega^n=
f_{\vep} \omega^n,
$$
where $c_{\vep}>0$ is a normalization constant. 
Since $e^{\rho_{\varepsilon}(\log f)}$ converges point-wise to $f$ on 
$X$ and since $e^{\rho_{\varepsilon}(\log f)} \leq 
\rho_{\varepsilon}(e^{\log f})$, by the {\it General Lebesgue 
Dominated Convergence Theorem} we see that 
$e^{\rho_{\varepsilon}(\log f)}$ converges to 
$f$ in $L^1(X)$ as $\varepsilon\to 0$. 
This implies that $c_{\varepsilon}$ converges 
to $1$ as $\varepsilon \to 0$. Then we can assume that 
$c_{\vep}\leq 2$. Thus we get the following
uniform control 
$$
f_{\vep} \leq e^{-\rho_{\vep}(\p^-) + C_2}.
$$
By Lemma \ref{lem: E stable} below we know that
$\f_{\vep}$ converges to $\f$ in $L^1(X)$. Thus the set 
$$
\Uc:= \{\f_{\vep}\ \setdef\ \vep>0\} \cup \{\f\}
$$ 
is compact in $L^1(X)$. Then it follows from the uniform Skoda integrability theorem (Lemma \ref{lem: Zer01} below) that for any $A>0$ we have
$$
\sup_{\vep>0}\int_X e^{-A\f_{\vep}}\omega^n <+\infty. 
$$
Thus, we can apply Theorem \ref{thm: C2 estimate} to find 
$C_3>0$ under control  such that
$$
\Delta_{\omega} \varphi_{\varepsilon} \leq C_3 e^{-2\p^-}.
$$

Fix  a compact $K\Subset X\setminus D$, $k\geq 2$ and $\beta\in(0,1)$. Now since  $0<f\in \Cc^{\infty}(X\setminus D)$ we have uniform controls on the derivatives of all orders of $\log f_{\varepsilon}$ on $K$. Using the standard Evans-Krylov method and Schauder estimates we then obtain
$$
\Vert\varphi_{\varepsilon}\Vert_{\Cc^{k,\beta}(K)} \leq C_{K,k,\beta} .
$$
This explains the smoothness of $\f$ on $X\setminus D$.

\begin{lem}\label{lem: E stable}
Let $(X,\omega)$ be a compact K\"ahler manifold of dimension $n$. Let
$(f_j)$ be a sequence of non-negative functions on $X$ such that 
$\int_X f_j\omega^n=\int_X \omega^n$. Assume that $f_j$ converges 
in $L^1(X)$ and point-wise to $f$. For each $j$, let $\varphi_j\in \EcX$ be the unique normalized solution to $\MA(\f_j)=f_j\omega^n$. 
Then $\f_j$ converges in $L^1(X)$ to $\f\in \EcX$ the unique normalized solution to $\MA(\f)=f\omega^n$.  
\end{lem}
\begin{proof}
We can assume that $\f_j$ converges in $L^1(X)$ to $\psi\in \psh(X,\omega)$. It follows from the Hartogs lemma that $\sup_X \psi=0$. For each $j\in \N$ set 
$$
\p_j:= \left(\sup_{k\geq j}\f_j\right)^*\ {\rm and}\ u_j:= \max(\psi_j, \f-1).
$$
Then we see that $\p_j\downarrow \psi$ and $u_j\downarrow u:=\max(\p,\f-1)\in \EcX$. We also have that $\sup_X u=0$. It follows from the comparison principle that 
$$
\MA(u_j) \geq \min\left(f,\inf_{k\geq j} f_k\right)\omega^n= g_j \omega^n.
$$ 
By the continuity of the Monge-Amp\`ere operator along decreasing sequences in $\EcX$ we get 
$$
\MA(u)= \lim_{j\to+\infty} \MA(u_j) \geq  \lim_{j\to+\infty} g_j \omega^n = f\omega^n.
$$ 
Then the equality holds since they have the same total mass. Finally, by the uniqueness result in the class $\EcX$ (see \cite{Di09}) we deduce that $u=\f$, which implies that $\p=\f$. The proof is thus complete.
\end{proof}

By \cite{GZ07}, functions in $\Ec(X,\omega)$ have zero Lelong number at every point on $X$. Thus the following lemma is a direct consequence of  the uniform Skoda integrability theorem due to Zeriahi \cite{Zer01}:
\begin{lem}\label{lem: Zer01}
Let $\Uc$ be a compact family of functions in $\Ec(X,\omega)$. Then for each $C_1>0$ there exists $C_2$ depending on $C_1$ and $\Uc$ such that
$$
\int_X e^{-C_1 \phi} \omega^n \leq C_2,
 \ \forall \phi\in \Uc.
$$
\end{lem}

\section{Asymptotic behavior near the divisor}
In Theorem \ref{thm: general uniform} we have given a very general 
$\Cc^0$ estimate. We only assumed that the density $f$ is bounded by $e^{-\phi}$ for some quasi plurisubharmonic function $\phi$, 
and there is no regularity assumption on $D$. 
It is therefore natural to investigate the asymptotic behavior of the solution near $D$ 
when we have more information about  $D$ and about the behavior of $f$ near $D$.  

Let $X$ be a compact K\"{a}hler manifold of dimension $n$ and let $\omega$ be a K\"ahler form on $X$. Let $D= \sum_{j=1}^{N} D_j$ be a simple normal crossing divisor on $X$. Here "simple normal crossing" means that around each intersection point of $k$ components $D_{j_1},...,D_{j_k}$ ($k\leq N$), we can find complex coordinates $z_1,...,z_n$ such that 
for each $l=1,...,k$ the hypersurface $D_{j_l}$ is locally given by $z_l=0$. For each $j$, let $L_j$ be the holomorphic line bundle defined by $D_j$. Let $s_j$ be a holomorphic section of $L_j$ defining $D_j$, i.e 
$D_j=\{s_j=0\}$. We fix  a hermitian metric $h_j$ on $L_j$ such that 
$\vert s_j\vert:= \vert s_j\vert_{h_j}\leq 1/e$. 

We say that $f$ satisfies Condition $\Sc(B,\alpha)$ for some $B>0,\alpha>0$ if
\begin{equation}\label{eq: cond f log bis}
f\leq \frac{B}{\prod_{j=1}^N \vert s_j\vert^2 (-\log \vert s_j\vert)^{1+\alpha}}.
\end{equation}

\subsection{Volume-capacity domination}\label{sect: vol cap}

\begin{lem}\label{lem: vol cap}
Assume that $f$ satisfies  (\ref{eq: cond f log bis}) for some 
$B>0,\alpha>0$. Then for each $0<\gamma<\alpha$ we can find $A>0$ which only depends on $B,\alpha,\gamma,\omega$ such that
$$
\vol_f(E):= \int_E f\omega^n \leq A\, \Capa_{\omega}(E)^{\gamma}, \;\forall E\subset X,
$$
where $\Capa_{\omega}$ is the Monge-Amp\`ere capacity introduced in \cite{Kol03}, \cite{GZ05}.
\end{lem}

Before giving the proof of the lemma, let us recall the definition and basic facts about Cegrell's classes. We refer the reader to \cite{Ce98, Ce04} for more details. 

Let $\Omega$ be a bounded hyperconvex domain in $\C^n$. 
The class $\Ec_0(\Omega)$  consists of bounded psh functions which vanish on the boundary and have finite total mass.

We say that $u\in\Ec^p(\Omega), p>0$  if there exists a sequence $(u_j)\subset \Ec_0(\Omega)$ decreasing to $u$ such that 
$$
\sup_j \int_{\Omega} (-u_j)^p(dd^c u_j)^n<+\infty.
$$

A function $u$ belongs to $\Fc(\Omega)$  if there exists a sequence $(u_j)\subset \Ec_0(\Omega)$ decreasing to $u$ such that 
$$
\sup_j \int_{\Omega} (dd^c u_j)^n<+\infty.
$$
We recall the local Monge-Amp\`ere capacity introduced in \cite{BT82}:
for any Borel subset $E\subset \Omega$, we define 
$$
\CapBT(E,\Omega):= \sup\left\{\int_E (dd^c u)^n\ \setdef\ u\in \psh(\Omega)\ , \ -1\leq u\leq 0\right\}.
$$
The relative extremal function of $E$ with respect to $\Omega$ is 
$$
u_{E,\Omega}:= \sup\left\{u \in \psh(\Omega)\ \setdef\ u\leq 0\ {\rm on}\ \Omega\ ,\ u\leq -1\ {\rm on}\ E\right\}.
$$
\begin{proof}[Proof of Lemma \ref{lem: vol cap}]

It follows from \cite{Kol03} that $\Capa_{\omega}$ is comparable to the local capacity $\CapBT(\cdot, \Omega)$, where $\Omega$ is an open subset contained in a local chart. By considering $E$ a small subset contained in a local chart we reduce the problem to showing that
\begin{equation}\label{eq: poincare 1}
\vol_g(E)\leq A_1 \CapBT(E,\D^n)^{\alpha} , \forall E\Subset \D^n_{\delta}\Subset \D^n ,
\end{equation}
where $\D^n$ is the unit polydisk in $\C^n$, $\delta>0$ small enough and fixed, and 
$$
g(z)= g(z_1,...,z_n):=\frac{1}{\prod_{j=1}^k\vert z_j\vert^2(1-\log \vert z_j\vert)^{1+\alpha}}, k\leq n.
$$
We prove (\ref{eq: poincare 1}) by induction using the ideas in \cite{ACK}. We start with the case $n=1$.

Set $E_r:=E \cap \partial\D_r$, for any $r\in[0,t]$. Define now $\tilde{E}:= \{r\in [0,t] \,|\, E_r \neq \emptyset \}$ and denote by $l(\tilde{E})$ the length of $\tilde{E}$. Since the function $r\mapsto \frac{1}{r(1-\log r)^{1+\alpha}}$ is non-increasing when $r$ is small,  we obtain
\begin{eqnarray*}
\int_E g(z) dV(z) &=&  \int_0^{2\pi} \int_{\tilde{E}} \frac{dr d\theta}{r(1-\log r)^{1+\alpha}} \\
 &\leq & 2\pi\int_0^{\ell({\tilde{E}})} \frac{dr}{r(1-\log r)^{1+\alpha}} \\
&\leq & \frac{C_1}{(-\log l(\tilde{E}))^{\alpha}}\\
&\leq & C_2 \,\left[\CapBT (E,\D)\right]^{\alpha},  
\end{eqnarray*}
where the last inequality follows from \cite[p.1336]{Kol94}.

Assume that the result holds for $n-1$. Let us prove it for $n$. 
Without loss of generality we can assume that $E$ is compact in $\D^n$. We can also assume that $k=n$ (if $k<n$ the situation is much easier).
Set $h=h_{E,\D^n}^*$ the relative extremal function of $E$. 
Consider
$$
g_n(w):=\frac{1}{\vert w\vert^2 (1-\log \vert w\vert)^{1+\alpha}} \ ,\
g_{n-1}(z):=\frac{1}{\prod_{j=1}^{n-1}\vert z_j\vert^2 (1-\log \vert z_j\vert)^{1+\alpha}}.
$$
For each $w\in \D$ set
$$
E_{w}=\{z\in \mathbb{D}^{n-1}\ \vert \ h(z,w)\leq -1\} \ \ {\rm and}\ \ h_w=h(\cdot, w).
$$
By induction hypothesis  we get 
$$
\vol_g (E) =\int_{\D} \vol_{g_{n-1}}(E_w) g_n(w)dV_2(w) \leq A_1\int_{\D} \left[\CapBT (E_w, \D^{n-1})\right]^{\gamma}g_n(w)dV(w).
$$

Fix now $w\in \D$ and denote by $u= h_{E_w,\D}^*$ the relative extremal function of $E_w$. Since $h\in \mathcal{F}(\D^n)$ it follows from \cite[Theorem 3.1]{ACK} that $h_w\in \Ec^1(\D^{n-1})$. We also have  $h_w\leq u$ and $h_w= -1$ on $E_w.$ Using integration by parts we get
\begin{eqnarray*}
\CapBT (E_w,\D^{n-1}) \leq  \int_{\D^{n-1}} (-h_w) (dd^c u)^{n-1} \leq  \int_{\D^{n-1}} (-h_w) (dd^c h_w)^{n-1}=: -\f(w).
\end{eqnarray*}
By \cite[Theorem 3.1]{ACK} we know that  $\varphi\in \Fc(\D)$. Moreover, we also have $\varphi \geq -A_0$ for some universal constant $A_0$ (here $A_0$ depends on $\delta$). Indeed, let $v$ be the relative extremal function of $\D_{\delta}^n$ with respect to $\D^n$. Since $h\geq v$, it is easy to see that for each $w\in \D$, $h_w \geq v_w$. From this we get a uniform lower bound for $\varphi$. Since $E$ is compact in $\D^n$ we also get 
$$
\mu=\int_{\D} dd^c \f = \int_{\D^n} (dd^c h)^n = \CapBT (E,\D^n).
$$
Thus, using the previous part (when $n=1$) we  obtain
\begin{eqnarray*}
\vol_g(E) & \leq & A_1 \int_{\D} (-\f(w))^{\gamma} g_n(w) dV_2(w) \\
& = & A_2 \int_0^{A_0} t^{\gamma-1} \vol_{g_n}(\f<-t) dt\\
& \leq & A_3 \int_0^{A_0}  t^{\gamma-\beta_1-1}\mu^{\beta_1} dt\\
& = & A_4 \left[\CapBT(E,\D^n)\right]^{\beta_1}.
\end{eqnarray*}
Here, we choose $\beta_1<\gamma$ so that the integrals converge. In the above we have used the fact that 
$$
\CapBT(v<-t) \leq \frac{1}{t}\int_{\D} dd^c v, \ \forall v\in 
\Fc(\D), \ \forall t>0.
$$
Since $\beta_1$ can be chosen arbitrarily near $\gamma$ (and the constant $A_4$ will increase), the result follows.
\end{proof}

When $\alpha=1$ we get the following  estimate.

\begin{lem}\label{lem: refined vol-cap domination alpha=1}
Let $\mu=f\omega^n$, $f=\frac{1}{\prod_{j=1}^N \vert s_j\vert^2 (-\log \vert s_j\vert)^{2}}$. Then there exists $A>0$  such that  for every Borel subset $E\subset X$ we have
\begin{equation}\label{eq: refined vol-cap domination}
\mu(E) \leq A\cdot  \left[\eta+(-\log \eta)^n\Capa_{\omega}(E)\right], \ \forall \eta\in (0,1/e).
\end{equation}
\end{lem}

\begin{proof}
We only give a sketch of the proof  since it is essentially a copy of the proof of Lemma \ref{lem: vol cap} with a  small change. We also use the same notation as there. Without loss of generality we can  assume that $E\Subset \D^n_{\delta}\Subset \D^n$ for some small fixed 
$\delta$. The function $\varphi$  
belongs to $\Fc(\D)$. The same arguments as in Lemma \ref{lem: vol cap} show that $\f$ is also bounded from below by $-A_1$ for 
some universal constant $A_1>0$. In the final step we get
\begin{eqnarray*}
\vol_g(E) & \leq & A_2 \int_{\D} \left(\eta+ (-\log \eta)^{n-1}(-\f(w))\right) g_n(w) dV_2(w) \\
& = & A_3 \eta + A_2(-\log \eta)^{n-1}\int_{0}^{A_1}  \vol_{g_n}(\f<-t) dt\\
&\leq & A_3 \eta + A_4 \eta^2 (-\log \eta)^{n-1}+ A_5(-\log \eta)^{n-1}\int_{\eta^2}^{A_1} \Capo(\varphi<-t) dt\\
&\leq & A_6 \eta + A_5(-\log \eta)^{n-1}\int_{\eta^2}^{A_1} 
\frac{1}{t}\left[\int_{\D} dd^c \varphi \right] dt\\
&\leq & A_6 \eta + A_7(-\log \eta)^{n}\int_{\D} dd^c \varphi .
\end{eqnarray*}
\end{proof}

\begin{lem}\label{lem: GZ exponential}
Let $\varphi\in \Ec(X,\omega)$ be such that $\sup_X \varphi=0$ 
and $\mu=\MA(\varphi)$ satisfies (\ref{eq: refined vol-cap domination}) for 
some $A>0$. Then there exists $C, c>0$ depending on $A$ such that 
$$
\Capo(\varphi<-t) \leq C e^{-ct},  \ \forall t>0,
$$
In particular, if $\beta<c$ then $\int_X e^{-\beta \varphi}d\mu \leq C'$, with $C'=C(\beta,A)>0$. 
\end{lem}
\begin{proof}
Fix  $s,t>1$. By standard application of the comparison principle we get
\begin{eqnarray}\label{eq: refined vol-cap domination 2}
\Capo(\varphi<-t-s) &\leq &\int_{\{\varphi<-t\}} \left(\omega+\frac{1}{s}dd^c \varphi\right)^n\\
&\leq & \frac{1}{s^n} \int_{\{\varphi<-t\}}  \sum_{k=0}^n C^k_n(s-1)^k\omega^k\wedge \omega_{\varphi}^{n-k} \nonumber \\
&\leq & \int_{\{\varphi<-t\}} \omega^n + \frac{2^n}{s} \int_{\{\varphi<-t\}} \MA(\varphi),\nonumber
\end{eqnarray}
where the last inequality follows from the partial comparison principle (see \cite[Theorem 2.3]{Di09}).  It follows from \cite{GZ05} that 
$$
\int_{\{\varphi<-t\}} \omega^n \leq C_1 e^{-at}, a>0.
$$
Choose $s:=2^{n}Ae$ and fix  $\varepsilon< \min(1,a,1/s)$.
Set 
$$
F(t):= \frac{e^{\varepsilon t}}{t^n}\Capo(\varphi<-t) ,\ t\geq 1.
$$
Now, if we choose $\eta=e^{-t}$ in  (\ref{eq: refined vol-cap domination}) and plug  (\ref{eq: refined vol-cap domination}) into (\ref{eq: refined vol-cap domination 2}) we get
$$
F(t^2+s) \leq C_2 + bF(t) , 
$$
where $b=2^nAe^{\varepsilon s}/s <1$. This yields $\sup_{t\geq 1} F(t) \leq C_3$, for some $C_3>0$ depending on $A$. We finally get 
$$
\Capo(\varphi<-t) \leq C e^{-ct}, c<\varepsilon. 
$$
The last statement easily follows since it follows from \cite[Lemma 2.3]{BGZ08} that
$$
\int_{\{\varphi<-t\}}\MA(\varphi)\leq t^n \Capo(\varphi<-t), 
\ \ \forall  t\geq 1.
$$
\end{proof}

\subsection{Proof of Theorem \ref{thm: main log}}\label{sect: proof main log}
Assume in this section that $f$ satisfies Condition $\Sc(B,\alpha)$
for some $B>0,\alpha>0$. The first part of Theorem \ref{thm: main log} was proved in Theorem \ref{thm: general uniform}. We divide the remaining parts into three cases depending on the value of $\alpha$.

\subsubsection{The case when $\alpha>1$} The continuity of $\f$ and the $\Cc^0$ estimate  follow directly from Lemma \ref{lem: vol cap} 
and Ko{\l}odziej's classical result (see \cite{Kol98}).

\subsubsection{The case when $0<\alpha<1$}
Fix $\beta\in (1-\alpha,1)$ and set $\delta=\alpha+\beta-1$, and 
$$
u_{\beta}:=\sum_{j=1}^N-a(-\log|s_j|)^{\beta},
$$
where $a>0$ is small enough so that $u_{\beta}\in \psh(X,\omega)$. 
By Theorem \ref{thm: general uniform} we have
$$
\varphi \geq \sum_{j=1}^N \log \vert s_j\vert -C_0,
$$
for some positive constant $C_0$ depending on $B$. By simple computations we obtain
$$
\MA(\varphi) \leq \frac{C_1f_{1-\beta}\omega^n}{(-\f)^{\delta}},
$$
for some positive constant $C_1$ depending on $C_0$. Here for each $r>0$, we set
$$
f_r:=\frac{1}{\prod_{j=1}^N|s_j|^2(-\log|s_j|)^{1+r}}.
$$
We also get 
$$
\MA(u_{\beta}-C_2) \geq \frac{C_1f_{1-\beta}\omega^n}{(-u_{\beta}+C_2)^{\delta}},
$$
where $C_2>0$ depends on $C_1,\delta$.
The comparison principle yields that $\f \geq u_{\beta}-C_2$.

\subsubsection{The case when $\alpha=1$}
Consider the model function
$$
\psi:= -A_1\sum_{j=1}^N \log(-\log \vert s_j\vert+A_2) ,
$$
where $A_1>0$ is big and $A_2$ is chosen so that $\psi$ is $\omega/2$-psh on $X$. It follows from 
Lemma \ref{lem: refined vol-cap domination alpha=1} 
and Lemma \ref{lem: GZ exponential} that $\int_X e^{-c\varphi} f\omega^n<C_1$ for some small constant $c>0$ depending on $B$. Here $C_1$ depends on $c$ and $B$.
Thus, for $t>0$, $p>1$,  by H\" older inequality we get
\begin{eqnarray*}
\int_{\{\varphi<\psi-t\}}\MA(\varphi) &\leq & \int_{\{\varphi<\psi-t\}}e^{-c \varphi/p} e^{c \psi/p} f\omega^n \\
&\leq & \left(\int_X e^{-c \varphi} f\omega^n\right)^{1/p} \left(\int_{\{\varphi<\psi-t\}} e^{c \psi/(p-1)} f\omega^n\right)^{1-1/p}\\
&\leq & C_2 \left(\Capis(\varphi<\psi-t)\right)^{\gamma},
\end{eqnarray*}
where $\gamma < A_1 c/p + (p-1)/p$ and $C_2>0$ is a universal constant. The last inequality follows from the volume-capacity 
domination (Lemma \ref{lem: vol cap}) and 
from Lemma \ref{lem: classical cap and cap psi}. 
Now if $A_1c>1$ 
we can choose $\gamma>1$ and the result follows as in Theorem \ref{thm: general uniform}.

\subsection{Regularity near the divisor $D$}\label{sect: asymp}
In this subsection we will discuss about the behavior of the solution to equation (\ref{eq: intro 1}) near the divisor $D$. We prove the following result when $\alpha<1$.
\begin{prop}\label{prop: asymptotic 1}
Consider $f= \frac{h}{\prod_{j=1}^N\vert s_j\vert^2(-\log \vert s_j\vert)^{1+\alpha}}$, where  $1/B\leq h\leq B$ on $X$ and $\alpha\in(0,1)$. Assume that $f$ is normalized so that $\int_X f\omega^n= \int_X \omega^n$. Let $\varphi\in \Ec(X,\omega)$ be the unique normalized solution of (\ref{eq: intro 1}). Then for each 
$0<p<1-\alpha$  and each $1-\alpha<q<1$,
we have
$$
-a_1(-\log \vert s\vert)^q -A_1\leq \varphi \leq -a_2(-\log \vert s\vert)^p +A_2,
$$
where $a_1,A_1>0$ depend on $B,\alpha, q$ while $a_2, A_2>0$ 
depend on $B,\alpha,p$.
In particular, the solution $\varphi$ goes to $-\infty$ on $D$.
\end{prop}
\begin{proof}
One inequality has been proved in Theorem \ref{thm: main log}. Let us prove the upper bound. We normalize $\varphi$ such that $\sup_X \varphi=-1$. To simplify the notation we denote, for each $r>0$, 
$$
f_r:= \frac{1}{\prod_{j=1}^N\vert s_j\vert^2(-\log \vert s_j\vert)^{1+r}} .
$$
Fix $p\in (0,1-\alpha)$ set $\delta:= (1-\alpha-p)/p >0$.
Consider $u_p:=-\sum_{j=1}^N a_2(-\log \vert s_j\vert)^p$, where $a_2>0$ is small so that $u_p$ is $\omega$-psh on $X$. Then we can find $C_3>0$ such that
$$
\MA(u_p) \leq \frac{C_3 f\omega^n}{(-u_p)^{\delta} },
$$
while since $\varphi\leq 0$, for some $A_2>0$ big enough (for instance $A_2^{\delta}=C_3$) we have
$$
\MA(\varphi-A_2) \geq \frac{C_3 f\omega^n}{(-\varphi+A_2)^{\delta} }.
$$
The comparison principle then yields the desired upper bound.
\end{proof}

By the same way we obtain a similar upper bound  when $\alpha=1$.
\begin{prop}\label{prop: asymptotic 2}
Assume that $f$ is normalized so that $\int_X f\omega^n= \int_X \omega^n$ and 
$$
f\geq  \frac{1}{B\prod_{j=1}^N\vert s_j\vert^2(-\log \vert s_j\vert)^{2}}.
$$
Let $\varphi\in \Ec(X,\omega)$ be the unique normalized solution of (\ref{eq: intro 1}). Then for any $p\in (0,1)$ there exist $a, A >0$ depending on $B,p$ such that 
$$
\varphi \leq -a \sum_j\left[\log (-\log \vert s_j\vert)\right]^p +A.
$$
In particular, $\f$ is not bounded and goes to $-\infty$ on $D$.
\end{prop}
\begin{proof}
The proof uses the same arguments as in Proposition 
\ref{prop: asymptotic 1}. 
\end{proof}

\section{The case of semipositive and big classes}\label{sect: semi}
In this section we prove Theorem \ref{thm: main semi}. For convenience let us recall the setting. We assume that $(X,\omega)$
is a compact K\"ahler manifold of dimension $n$ and $D$ is an
arbitrary divisor on $X$.  
Let $E= \sum_{j=1}^M a_jE_j$ be an effective  snc divisor on $X$.  Let $\theta$ be a smooth  semipositive form on $X$ such that $\int_X \theta^n>0$ and 
$\{\theta\}-c_1(E)$ is ample. Consider the following degenerate complex Monge-Amp\`ere equation
\begin{equation}\label{eq: MAeq semipositive bis}
(\theta + dd^c \varphi)^n = f\omega^n,
\end{equation}
where $0\leq f\in L^1(X,\omega^n)$ satisfies the compatibility 
condition $\int_X f\omega^n = \int_X \theta^n$. 

For each $j=1,...,M$ let  $K_j$ be the holomorphic line bundle defined by $E_j$. Let 
$\sigma_j$ be a holomorphic section of  $K_j$ that vanish on 
$E_j$. We fix hermitian metric $h_j$ on $K_j$ such that $\vert \sigma_j\vert \leq 1/e$. Since $\{\theta\}-c_1(E)$ is ample, we can assume that 
$$
\theta+dd^c \phi=\omega_0+[E],
$$ 
where $\omega_0$ is a K\"{a}hler form on $X$ and 
$$
\phi:=\sum_{j=1}^M a_j\log\vert \sigma_j\vert.
$$
By rescaling $\omega$ we can also assume that $\omega_0\geq \omega$. Recall that $f$ satisfies Condition $\Hc_f$ on $X$, i.e. there is 
a constant $C>0$ such that
\begin{equation}\label{eq: cond Hf semi}
f=e^{\p^+-\p^-}, \ \ dd^c \p^{\pm} \geq -C\omega ,\ \sup_X \p^+\leq C,\ \ \p^-\in L^{\infty}_{\rm loc}(X\setminus D).
\end{equation}

\subsection{Uniform estimate}
The following $\Cc^0$-lower bound can be proved in the same ways as we have done in Theorem \ref{thm: general uniform}:
\begin{thm}\label{thm: uniform poly semi}
Assume that $D,E$ and $\theta$ are as above and $f$ satisfies (\ref{eq: cond Hf semi}). Let $\varphi$ be  the unique normalized solution to 
equation (\ref{eq: MAeq semipositive bis}). Then $\varphi$ is uniformly bounded away from $D\cup E$. More precisely, for any $a>0$ there exists $A>0$ depending on $C$ and $\int_X e^{-2\f/a}\omega^n$ such that
$$
\varphi \geq a \p^- +\phi -A .
$$
\end{thm}
\begin{proof}
It suffices to prove the result for small $a>0$. Fix $a>0$ very small
so that 
$$
\p:= a\p^- +\frac{1}{2}\phi \in \psh(X,\theta/2).
$$
It follows from Proposition 3.1 in \cite{EGZ09} that 
$$
\vol_{\omega} \leq C_1 \exp\left(\frac{-C_2}{\left[\Capa_{\theta/2}\right]^{1/n}} \right),
$$
for some universal constants $C_1,C_2>0$. Now, the same proof of Lemma  \ref{lem: classical cap and cap psi} yields
$$
\Capa_{\theta/2} \leq \Capis,
$$
where $\Capis$ is the generalized capacity defined by
the form $\theta$ and $\psi$:
$$
\Capis(E):= \sup\left\{ \int_E  (\theta+dd^cu)^n \ \setdef\ u\in \psh(X, \theta) , \ \psi-1\leq u\leq \psi\right\}.
$$
Then we can repeat the arguments in the proof of Theorem  \ref{thm: general uniform} to get the result.
\end{proof}

\subsection{Laplacian estimate}
We now prove a $C^2$ a priori estimate in the semipositive and big case. Even when $f$ is smooth on $X$, $\f$ is only smooth in the ample locus of $\theta$. To get rid of this, we replace $\theta$ by 
$\theta+ t\omega$, $t>0$. In principle, the $\Cc^2$ estimate will depends heavily on $t>0$ and we will have serious problem when $t\downarrow 0$. But, fortunately, the so-called Tsuji's trick (see \cite{Ts98}) allows 
us to get around this difficulty. In the sequel, we follow essentially 
the ideas in \cite{BEGZ10}. 

\begin{thm}\label{thm: C2 estimate_semi}
Let $f=e^{\psi^+-\psi^-}$ where $\psi^+,\psi^-$ are smooth on $X$. 
Fix $t\in(0,1)$. Let $\varphi\in \Cc^{\infty}(X)$ be the unique normalized solution to
$$
(\theta+ t\omega+ dd^c \varphi)^n=e^{\psi^+-\psi^-}\omega^n.
$$
Assume given a constant  $C>0$ such that
$$
dd^c \psi^{\pm}\geq -C\omega , \ \  \sup_X \psi^+\leq C .
$$
Assume also that the holomorphic bisectional curvature of $\omega$ is
bounded from below by $-C$. 
Then there exists $A>0$ depending on $C$ and $\int_X e^{-2(4C+1)\f}\omega^n$ such that 
$$
\Delta_\omega \varphi \leq A e^{-2\psi^--(4C+1)\phi}.
$$
\end{thm}
\begin{proof}
Ignoring the dependence on $t$, we denote 
$\omega_{\varphi}:=\theta+t\omega+dd^c\varphi$. Consider the following
function
$$
H:= \log \tr_{\omega}(\omega_{\varphi}) + 2\psi^- 
- (4C+1)(\varphi-\phi),
$$
Since $\phi$ goes to $-\infty$ on $E$, 
we see that $H$ attains its maximum on $X\setminus E$ at some point 
$x_0\in X\setminus E$. From now on we carry all computations on 
$X\setminus E$. 
We can argue as in Theorem \ref{thm: C2 estimate} to obtain
\begin{eqnarray}\label{eq: C2 estimate semi 1}
\Delta_{\omega_\varphi}\log \tr_{\omega}(\omega_\varphi)\geq -3C\tr_{\omega_\varphi}(\omega)-\Delta_{\omega_\f}\psi^-.
\end{eqnarray}
Since $\omega_0+t\omega \geq \omega$ we get
\begin{equation}\label{eq: C2 estimate semi 2}
\Delta_{\omega_\varphi}(\varphi-\phi) \leq \tr_{\omega_{\varphi}}
(\omega_{\varphi}-\omega_0-t\omega) \leq n -\tr_{\omega_{\varphi}}
(\omega).
\end{equation}
Therefore,  from (\ref{eq: C2 estimate semi 1}) and 
(\ref{eq: C2 estimate semi 2}) we deduce that on $X\setminus E$
\begin{eqnarray*}
\Delta_{\omega_\varphi}H &\geq & \tr_{\omega_\varphi}(\omega)-n (4C+1) .
\end{eqnarray*}
We now apply the maximum principle to the function $H$ at $x_0$:
$$
0\geq \Delta_{\omega_\f} H(x_0)\geq \tr_{\omega_\varphi}(\omega)(x_0)-n (4C+1).
$$ 
Furthermore, by Lemma \ref{lemm1} we get
$$
\tr_{\omega}(\omega_\f)(x_0)\leq n e^{\psi^+-\psi^-}(x_0) \left(\tr_{\omega_\varphi}(\omega)\right)^{n-1}(x_0)\leq A_1 e^{\psi^+-\psi^-}(x_0) ,
$$ 
and hence, since $\sup_X \psi^+\leq C$,
$$
\log \tr_{\omega}(\omega_\f)(x_0)\leq \log A_1 +\psi^+(x_0)-\psi^-(x_0)\leq A_2-\psi^-(x_0)\, .
$$
It follows that
$$
H(x)\leq H(x_0)\leq A_2 +\p^-(x_0) -(4C+1)(\varphi-\phi)(x_0).
$$ 
By assumption and the $\Cc^0$ estimate in Theorem 
\ref{thm: uniform poly semi} we have 
$$
\f\geq \frac{1}{4C+1}\p^- + \phi -A_3,
$$
where  $A_3$ depends on $C$ and $\int_X e^{-2(4C+1)\f}\omega^n$. Thus
$$
\log \tr_{\omega}(\omega_\f)\leq A_4-2\psi^- +(4C+1)(\f-\phi).
$$
We finally get
$$
\tr_{\omega}(\omega_{\f})\leq A_5 e^{-2\psi^- -(4C+1)\phi}.
$$
\end{proof}
\begin{proof}[Proof of Theorem \ref{thm: main semi}]
We proceed as in Section \ref{sect: proof main general}. We also borrow the 
notations there. Let $\rho_{\varepsilon}(\psi^{\pm})$ be the Demailly's smoothing regularization of $\psi^{\pm}$. For each $\varepsilon>0$ let $\varphi_{\varepsilon}$ be the unique smooth 
function such that $\sup_X \f_{\varepsilon}=0$ and 
$$
(\theta+\vep \omega +dd^c \f_{\vep})^n = 
c_{\vep} e^{\rho_{\vep}(\psi^+)-\rho_{\vep}(\psi^-)}\omega^n,
$$
where $c_{\vep}$ is a normalization constant. As in Section \ref{sect: proof main general} we have a uniform control on the right-hand side:
$$
c_{\varepsilon} e^{\rho_{\varepsilon}(\psi^+)-\rho_{\varepsilon}
(\psi^-)} \leq e^{C-\p^-_{\vep}} .
$$
Now, we can copy the arguments in Section \ref{sect: proof main general} since our uniform estimate and laplacian estimate do not depend on $\vep$.
The proof is thus complete.
\end{proof}

\end{document}